\documentclass[review,hidelinks,onefignum,onetabnum]{siamart220329}
\usepackage{lineno}
\nolinenumbers

\usepackage{amsfonts}
\usepackage{graphicx}
\usepackage{epstopdf}
\usepackage{algorithmic}
\ifpdf
  \DeclareGraphicsExtensions{.eps,.pdf,.png,.jpg}
\else
  \DeclareGraphicsExtensions{.eps}
\fi

\newsiamremark{hypothesis}{Hypothesis}
\crefname{hypothesis}{Hypothesis}{Hypotheses}
\newsiamthm{claim}{Claim}

\headers{Approximation of  Sums of Locally Dependent Random Variables }{Zhonggen Su, Vladimir V. Ulyanov and Xiaolin Wang}

\title{Approximation of  Sums of Locally Dependent Random Variables via Perturbation of Stein Operator\thanks{Submitted to the editors DATE.
\funding{This work was funded by the NNSF of China (12271475, U23A2064, 11731012,) and Fundamental Research Funds for Central Universities grant(2020XZZX002-03).}}}

\author{Zhonggen Su\thanks{School of Mathematical Sciences, Zhejiang University, Hangzhou, 310058,  China.
  (\email{suzhonggen@zju.edu.cn})}
\and Vladimir V. Ulyanov\thanks{Faculty of Computational Mathematics and Cybernetics, Lomonosov Moscow State University, Moscow, 119991, Russia\\ Institute for Financial Studies, Shandong University, Jinan, 250100,  China.
	(\email{vulyanov@cs.msu.ru})}
\and Xiaolin Wang\thanks{School of Mathematical Sciences, Zhejiang University, Hangzhou, 310058,  China.
  (\email{12035020@zju.edu.cn})}}

\usepackage{amsopn}

\ifpdf
\hypersetup{
  pdftitle={Approximation of  Sums of Locally Dependent Random Variables via Perturbation of Stein Operator},
  pdfauthor={Zhonggen Su, Vladimir V. Ulyanov, and Xiaolin Wang}
}
\fi

\begin{document}

\maketitle

\begin{abstract}
Let $(X_{i}, i\in J)$ be a  family of locally dependent nonnegative integer-valued random variables, and consider the sum $W=\sum\nolimits_{i\in J}X_i$. We first establish a  general error upper bound for $d_{TV}(W, M)$ using Stein's method, where the target variable $M$ is either the mixture of Poisson distribution and binomial or negative binomial distribution. As applications, we attain $O(|J|^{-1})$ error bounds  for ($k_{1},k_{2}$)-runs and $k$-runs  under some special cases. Our results are significant improvements of the existing results in literature, say $O(|J|^{-0.5})$ in Pek\"{o}z [Bernoulli, 19 (2013)] and $O(1)$ in Upadhye, et al. [Bernoulli, 23 (2017)].
\end{abstract}

\section{Introduction}
Stein \cite{stein1972bound} first introduced a new method, now referred to as Stein's method, to obtain the bound for the departure of the distribution of the sum of $n$ terms of a stationary random sequence from a normal distribution. Soon after,  Chen \cite{chen1975Poisson} further developed  Stein's method to establish a Poisson approximation for dependent trials. Since these seminal papers, there have been considerable developments and many applications in apparently distinct fields.  In fact, The Stein method is suitable for not only normal and Poisson approximations, but also exponential distributions \cite{chatterjee2011exponential,pekoz2011new}, geometric distributions \cite{pekoz1996stein,pekoz2013total}, chi-square distributions \cite{mann1997stein}, general gamma distributions \cite{luk1994stein,pekoz2016generalized}, etc. The interested reader is referred to \cite{arratia1989two,barbour1992Poisson,chen2011normal,ross2011fundamentals,nourdin2012normal} for its history and applications, and \cite{nourdin2012normal,chen2011normal,daly2012stein,ley2013stein,ley2013local,ley2014approximate} and their references for recent progresses.

Based on a number of specific distribution approximations, Ley et al. \cite{ley2014approximate} proposed a canonical Stein operator for both continuous and discrete distributions,  and the resulting Stein identity highlighted the unifying theme behind the literature on Stein's method. Let $Y$ be a random variable with a target distribution. An operator $\mathcal{A}_{Y}$ is called a canonical Stein operator of $Y $ for a function class  $\mathcal{G}$ if   $\mathbf{E}(\mathcal{A}_{Y}f(Y))=0$  for each $f\in \mathcal{G}$.

According to this definition, for $Y\sim N(0, 1)$, $\mathcal{A}_{Y}f(x)=f'(x)-xf(x)$ for any  $f\in \mathcal{G}$ where $\mathcal{G}$ is the set of all bounded differentiable functions on $\mathbf{R}$; while for $Y\sim \mathcal{P}(\lambda)$, $\mathcal{A}_{Y}f(k)= kf(k)-\lambda f(k+1)$ for any  $f\in \mathcal{G}$ where $\mathcal{G}$ is the set of all bounded  functions defined  on $\mathbf{Z}$ with $f(0)=0$.

The purpose of this paper is to establish a sharper approximation of the sum of discrete dependent random variables by a three-parameter distribution which is a small perturbation to Bernoulli and negative binomial distributions.

Let $J$ be a finite subset of $\mathbf{N}$, $ \left\{X_{i}, i \in J\right\} $ a sequence of random variables  taking values in $\mathbf{N}$.  For a given set $A\subset J$,  $X_{A}$ denotes the collection of random variables $\{X_{i}, i \in A\}$,  $X_{A}^{*} := \sum\nolimits_{i \in A}  X_{i}$. As usual, $A^c$ denotes the complement of the set $A$.  Call  $ \left\{X_{i}, i \in J\right\} $  a  local  dependence structure if for any $i\in J$, there exist  $A_i$, $B_i$ and $C_i$ such that $i \in A_{i} \subseteq B_{i} \subseteq C_{i} \subset J$ and  $X_{i}$ is independent of $X_{A_{i}^{c}}$,  $X_{A_{i}}$ is independent of $X_{B_{i}^{c}}$ and $X_{B_{i}}$ is independent of $X_{C_{i}^{c}}$.  This local dependence structure has been well-studied in  literature, say Chen and Shao \cite{chen2004normal}, in which it was referred to as Assumption (LD3). The object of our study is defined by
\begin{align}\label{ss}
W =\sum\nolimits_{i \in J}  X_{i}.
\end{align}
We aim to find as good an approximation as possible for $W$ when $\mathbf{E}W^{3}< \infty$. The target distribution is either $M_1=B(n, p)\ast \mathcal{P}(\lambda)$ or $M_2=NB(r, \bar{p})\ast \mathcal{P}(\lambda)$ with same expectation $\mu$ and variance $\sigma^{2}$ of $W$.  In fact, there has been a lot of research work on approximating   $W$, say  Vellaisamy et al.  \cite{vellaisamy2013negative},   Upadhye et al. \cite{upadhye2017Stein}.

To state our main results, we need some additional notation. Define
\begin{align}\label{theta}
    \theta_{1} &=\frac{\lambda}{\lfloor n+\lambda/p\rfloor q},\; \theta_{2} =\frac{\lambda \bar{q}}{(r\bar{q}+\lambda\bar{p})},\;  \Theta_{1} = \frac{\theta_1 q} { (1-2\theta_{1}) \lambda p }, \; \Theta_{2} = \frac{\theta_2} { (1-2\theta_{2})\lambda \bar{q} },\\\label{10013}
 D&(W|X_{C_i}) =\left\|\mathcal{L}(W|X_{C_i}) *\left(I_{1}-I_{0}\right)^{* 2}\right\|_{\mathrm{TV}},\quad \text{for each}\; i\in J,
\end{align}
where  $I_{1}$ and $I_{0}$ denote the degenerate distributions concentrated at 1 and 0 respectively, and $\mathcal{L}(W|X_{C_i})$ stands for the conditional distribution of $W$ conditioned upon $X_{C_i}$.  Denote
\small$$\Xi_{i,j}=9\sum\nolimits_{(\mathbf{E})}[(1-\beta_{j})\mathbf{E}X_{i}(\mathbf{E})X_{A_{i}}^{*}(\mathbf{E})X_{B_{i}}^{*}(\mathbf{E})X_{C_{i}}^{*}
    +\beta_{j}\mathbf{E}X_{i}(\mathbf{E})X_{B_{i}}^{*}(\mathbf{E})X_{C_{i}}^{*})] \operatorname{ess}D(W\mid X_{C_{i}}),$$
    \normalsize
    where $\beta_{1}=-p/q,$ $\beta_{2}=\bar{q}$, $\operatorname{ess}X$ denotes the essential supremum of $X$ and $\sum\nolimits_{(\mathbf{E})}$ denotes the sum over all possible choices of $\mathbf{E}$ in front of   $X$'s, $(\mathbf{E})X_{i}$ stands for $X_{i}$ or $\mathbf{E}X_{i}$.

We also need to make some basic assumptions about the parameters of the target distributions.

\noindent \textbf{(H1)}: The triple $\{n,p,\lambda\}$ satisfies
$ \varepsilon_{1} := \lambda p/q <p(\lfloor n+\lambda/p\rfloor)/2, q=1-p .$

\noindent\textbf{(H2)}: The triple $\{r,\bar{p},\lambda\}$ satisfies
$\varepsilon_{2} :=\lambda \bar{q} < (r\bar{q}+\lambda \bar{p})/2, \bar{q}=1-\bar{p}.$

Our main results read as follows.
\begin{theorem}\label{main}
With the above notation and assumptions, we have

(i)\;if\; $\mathbf{E}W < \operatorname{Var}W$, then
\begin{align}\label{i}
    d_{T V}(W, M_{2}) \leq \Theta_{2}\sum\nolimits_{i \in J}\Xi_{i,2};
\end{align}

(ii)\;if\; $\mathbf{E}W > \operatorname{Var}W$, then
\begin{align}\label{ii}
    d_{T V}(W, M_{1}) \leq \Theta_{1}\left[\sum\nolimits_{i \in J}\Xi_{i,1} +\frac{\delta p^{2}}{q}\right].
\end{align}
\end{theorem}

Note that $M_{1}$ and $M_{2}$ can be represented as sums of independent
identical random variables. Indeed,   $M_{1}=\sum\nolimits_{1}^{|J|}\xi_{i,1}$ with   $\xi_{i,1}\sim B(n/|J|,p)*\mathcal{P}(\lambda/|J|)$,  and $M_{2}=\sum\nolimits_{1}^{|J|}\xi_{i,2}$ with   $\xi_{i,2}\sim NB(r/|J|,\bar{p})*\mathcal{P}(\lambda/|J|)$. A recent result in Bobkov and Ulyanov \cite{bobkov2022chebyshev} yields a refined central limit theorem.
 Denote by  $\mathbf{\Phi}$ the standard normal distribution function, $\mathbf{\phi}$ the standard normal density function
\begin{align*}
     l_{3,j}&=\frac{1}{\sigma^{3}}\mathbf{E}\mathbf(M_{j}-\mathbf{E}M_{j})^{3},\quad \mathbf{\Phi}_{3,j}(x)=\mathbf{\Phi}(x)-\frac{l_{3}}{6}(x^{2}-1)\mathbf{\phi}(x),\quad  x\in \mathbf{R},\\
     L_{4,j}&=\frac{1}{\sigma^{4}}\sum_{i\in J}\mathbf{E}(\xi_{i,j}-\mathbf{E}\xi_{i,j})^{4},\quad   V_{j}=-\sum_{i=1}^{|J|}\sup_{0\leq t \leq 2\pi}\frac{|\mathbf{E}e^{it\xi_{i,j}}|}{1-\cos t},\quad j=1,2.\end{align*}
Applying Theorem 4.1 in \cite{bobkov2022chebyshev}, we obtain
\begin{align}\label{Com1}
    \sup_{k\in \mathbf{Z}}\left|\mathbf{P}(M_{j}\leq k)-\mathbf{\Phi}_{3,j}(\frac{k+1/2-\mu}{\sigma})\right| \leq \frac{\sigma^{2}L_{4,j}}{V_{j}}.
\end{align}
Combing \eqref{Com1} and Theorem \ref{main} yields
\begin{theorem}\label{Main2}
      Denote $K_{J}=\sup_{i\in J}|A_{i}||B_{i}||C_{i}|$. Assume further that \\$D(W\mid X_{C_{i}})\leq C/\sigma^{2}$. Then for  $j=1, 2$, we have
     \begin{align}\label{TM1}
         \sup_{k\in \mathbf{Z}}\left|\mathbf{P}(W\leq k)-\mathbf{\Phi}_{3,j}(\frac{k+1/2-\mu}{\sigma})\right|\leq C\left[\frac{\Theta_{i}K_{J}\sum\nolimits_{i\in J}\mathbf{E}X_{i}^{3}}{\sigma^{2}}+\frac{\sigma^{2}L_{4,j}}{V_{j}}\right].
     \end{align}
Moreover, if  $\lim_{|J|\rightarrow \infty}K_{J}$ and all $\mathbf{E}X_{i}^{3}$ are bounded by a constant, then
\begin{align}\label{TM2}
    \sup_{k\in \mathbf{Z}}\left|\mathbf{P}(W\leq k)-\mathbf{\Phi}_{3}(\frac{k+1/2-\mu}{\sigma})\right|\leq C\frac{|J|}{\sigma^{4}\wedge (\sigma^{2}\mu)}.
\end{align}
\end{theorem}

The rest of the paper is organized as follows.   Section \ref{S2}   first obtain the Stein operators for   $M_1$ and $M_2$,  and then express them as the perturbation of some classic Stein operators. We also determine the parameters of $M_1$ and $M_2$ by matching their first three factorial cumulants and those of $W$.  In Section \ref{S3}, we prove Theorem \ref{main}.  A key ingredient is to control  $|\mathbf{E}g(W)-\mathbf{E}g(M_{i})|$   by using the third-order difference expansion of $g$.    Section \ref{S4} is devoted to the application of main results to the runs of independent  Bernoulli trials. Sharper upper error bounds are obtained by explicitly computing the factorial cumulants and estimating the measure of smoothness.

\section{Stochastic perturbation tricks}\label{S2}

Let us start with the  Stein operator associated with a specific nonnegative integer-valued random variable.  It is sometimes easy to figure out explicitly. Indeed, let $Y$ be a  nonnegative random variable that takes values in $\mathbf{N}$ with
$\mu_{k}=\mathbf{P}(Y=k)>0.$
Then the stein operator $\mathcal{A}_{Y}$  can be defined as
 $\mathcal{A}_{Y} g(k)=(k+1) \mu_{k+1}g(k+1)/\mu_{k}-k g(k), k \in \mathbf{Z}_{+}$, namely $\mathbf{E}\mathcal{A}_{Y} g(Y)=0$.
Below are three basic examples, which are frequently used throughout the paper.\\
$\bullet$\;\textit{Let $\xi_1 \sim B(n,p)$, $n \in \mathbf{Z}_{+}$, $p \in(0,1)$. Namely,
$$ \mathbf{P}\left(\xi_{1}=k\right)= \binom n k  p^{k} q^{n-k},  k=0,\cdots, n, q=1-p.$$
The Stein operator for $\xi_1$ is $\mathcal{A}_{\xi_1} g(k)=(n - k)p g(k+1)/q- k g(k),\;k\in \mathbf{Z}^{+}.$}\\
$\bullet$\;\textit{Let $\xi_2\sim NB(r, \bar{p})$, $r>0$, $\bar{p}\in (0,1)$. Namely,
$$\mathbf{P}(\xi_2=k)=\binom {r+k-1}  k  \bar{p}^{r} \bar{q}^{k}, k\in \mathbf{N}, \bar{q}=1-\bar{p}.$$
The Stein operator for $\xi_2$ is $\mathcal{A}_{\xi_2} g (k)=\bar{q}(r+k) g(k+1)-k g(k),\; k\in \mathbf{Z}^{+}.$}\\
$\bullet$\;\textit{Fix $N \in (1,+\infty)$ (not necessarily an integer), $0<p<1$. Let $\xi_3 \sim PB(N,p)$  the pseudo-binomial distribution (see  \v{C}ekanavi\v{c}ius and Roos \cite{cekanavicius2004two}, p. 370). Namely,}
$$\mathbf{P}\left(\xi_3=k\right)=\frac{1}{C_{N, p}}\binom N  k  p^{k} q^{N-k}, \quad k=0,\cdots,\lfloor N\rfloor$$
 \textit{where $ q = 1- p$, $C_{N, p}$ is the normalization constant. The Stein operator for $\xi_3$ is}
 \begin{align*}
 \mathcal{A}_{\xi_3} g(k)=(Np/q-p k/q) g(k+1)-k g(k), \quad k=1,2, \cdots,\lfloor N\rfloor.
\end{align*}
As the reader notice,  these three  operators can be expressed in a unified way,
\begin{align}\label{per}
	&\mathcal{A}_{\xi_{i}} g(k)=(\alpha_{i}+\beta_{i} k) g(k+1)-k g(k), \quad k \in \mathbf{Z}_{+},\quad\\
\text{where}\quad\quad&\notag\\
\label{1.99}
   &\alpha_{1}=np/q,\quad \beta_{1}=-p/q;\quad\alpha_{2}=r \bar{q},\quad \beta_{2}=\bar{q};\quad \alpha_{3}=Np/q,\quad \beta_{3}=-p/q.
\end{align}

 Next,  we turn to the Stein operator for the sum of two independent random variables. The basic tool is the probability generating function  approach. Lemma \ref{r} can be shown in a completely parallel way to that of  Proposition 2.2 in \cite{upadhye2017Stein}.

\begin{lemma}\label{r} Fix $n, p, \lambda, r, \bar{p}$. Let $\xi_{i}$ and $\eta$ be two independent random variables, $\eta \sim \mathcal{P}(\lambda)$. Assume further that $\xi_{1}\sim B(n, p)$ or  $\xi_{2}\sim NB(r, \bar{p})$. Then  $M_{i}:=\xi_{i}+\eta$ has a Stein operator of the form
     \begin{align}\label{oper}
		\mathcal{A}_{M_{i}} g(k)=(a_{i}+\beta_{i} k) g(k+1)-k g(k) - \lambda \beta_{i} \Delta g(k+1),\quad i=1,2
	 \end{align}
	 where $a_{i}=\alpha_{i}+\lambda \bar{\beta_{i}}$ and $\alpha_{i}$, $\beta_{i}$ were given by (\ref{1.99}) accordingly.
 \end{lemma}
Comparing (\ref{per}) and  (\ref{oper}), it is easy to give the Stein operators of $M_{1}$ and $M_{2}$.
\begin{proposition}Let $\mathcal{U}_{1}(g)(k)=\lambda \Delta g(k+1)p/q$ and $\mathcal{U}_{2}(g)(k)=-\lambda \bar{q} \Delta g(k+1)$.

(1) Fix $n\ge 1, 0<p<1, \lambda>0$. Let $N= n+\lambda/p$, $\zeta_{1}\sim PB(N, p)$, then
\begin{align}\label{PoPB}
\mathcal{A}_{M_{1}}=\mathcal{A}_{\zeta_{1}}+ \mathcal{U}_{1},
\end{align}

(2) Fix $ r>0, 0<\bar{p}<1, \lambda>0$. Let  $\zeta_{2}\sim NB(r+\lambda \bar{p}/\bar{q}, \bar{p})$, then
\begin{align} \label{PoNB}
\mathcal{A}_{M_{2}}=\mathcal{A}_{\zeta_{2}} +\mathcal{U}_{2}.
\end{align}
\end{proposition}

Having the Stein operator $\mathcal{A}_{M_{i}}$, we take a close look at  the properties of solution $g_{f}$ to the following equation
\begin{align}\label{so}
    \mathcal{A}_{M_{i}}g(k)=f(k)-\mathbf{E}f(M_{i}), \quad f \in \mathcal{G}.
\end{align}
It follows from Lemma 2.2 in \cite{barbour1987asymptotic},  Lemma 9.2.1 in \cite{barbour1992Poisson} and (57) of \cite{cekanavicius2004two}  that
\begin{align}\label{2.3}
     \|\Delta g_{f}^{\xi_{1}}\| \leq \frac{2\|f\|}{np},\quad  \|\Delta g_{f}^{\xi_{2}}\| \leq \frac{2\|f\|}{r \bar{q}}\quad\text{and}\quad \|\Delta g_{f}^{\xi_{3}}\| \leq  \frac{2\|f\|}{\lfloor N\rfloor p},
\end{align}
 where $g_{f}^{\xi_{i}}$ is the solution to  the corresponding Stein equation.

In our context, we are mainly concerned with the Stein operators in (\ref{PoPB}) and (\ref{PoNB}). Note $\|\mathcal{U}_{1} g\| \leq \varepsilon_{1}\|\Delta g\|$ and $\|\mathcal{U}_{2} g\| \leq \varepsilon_{2}\|\Delta g\|$
   where  $\varepsilon_{1} =  \lambda p/q$ and  $\varepsilon_{2}= \lambda \bar{q} $. The following lemma due to Barbour et  al.\cite{barbour2007Stein} offers an upper bound for $\|\Delta g^{M_{i}}_{h}\|$ which is the solution of (\ref{so}) when we restrict the domain of Stein operator $\mathcal{A}_{M_{i}}$  to $\mathcal{H}$, the set of indicator functions.

\begin{lemma}\label{lemma}With \textbf{(H1)} and \textbf{(H2)}, we have
\begin{align*}
	\|\Delta g_{h}^{M_{1}}\| \leq \frac{1}{ \lfloor n+\frac{\lambda}{p}\rfloor p -2\varepsilon_{1}},\;h \in \mathcal{H}; \quad  \|\Delta g_{h}^{M_{2}} \|\leq \frac{1}{r\bar{q}+\lambda\bar{p} -2\varepsilon_{2}},\; h \in \mathcal{H}.
  \end{align*}
\end{lemma}
The proof is similar to Lemma 3.1 of \cite{upadhye2017Stein} with some minor modifications. The interested reader is referred to  \cite{barbour2007Stein} for a general framework.

Next turn to  the primary object of study, $d_{TV}(W, M_{i}), i=1,2.$ Note that
$$ d_{TV}(W, M_{i})= \sup_{h\in \mathcal{H}} \left|\mathbf{E}\left(\mathcal{A}_{M_{i}}g_{h}^{M_{i}}\right)(W)\right|.$$
Applying Lemma \ref{lemma}, we have
\begin{proposition}\label{lemmaWM}
 \noindent
	(i)\; Assume \textbf{(H1)}  and $\left|\mathbf{E}\left(\mathcal{A}_{M_{1}} g_{h}^{M_{1}}\right)(W)\right| \leq \varepsilon \|\Delta g_{h}^{M_{1}}\|$, then
	\begin{align}\label{999}
    d_{T V}(W, M_{1}) \leq \frac{\varepsilon q}{\lfloor n+\lambda/p\rfloor pq-2 \lambda p}.
\end{align}
  (ii)\; Assume \textbf{(H2)} and
$\left|\mathbf{E}\left(\mathcal{A}_{M_{2}} g_{h}^{M_{2}}\right)(W)\right| \leq \varepsilon \|\Delta g_{h}^{M_{2}}\|$, then
   \begin{align}\label{888}
    d_{T V}(W, M_{2}) \leq \frac{\varepsilon }{(r\bar{q}+\lambda\bar{p})-2 \lambda \bar{q}}.
\end{align}
\end{proposition}

 To conclude this section, we would like briefly explain  how to decide  the parameters ($\lambda, n, p$) and  ($\lambda, r, \bar{p}$). Its basic principle is as follows.

  Denote by  $\Gamma_{j}$ the $j$-th factorial  cumulant, namely,
    $ \Gamma_{j}(\cdot)=d^{j} (\log \phi_{\cdot}(z)|_{z=1})/d z^{j}$, where $\phi_{X}(z)$ is the generating function of some random variable $X$. Then by some basic algebra, we obtain
\begin{align*}
\Gamma_{1}(M_{i})= \frac{\alpha_{i}}{1-\beta_{i}}+\lambda,\quad \Gamma_{2}(M_{i})= \frac{\alpha_{i}\beta_{i}}{(1-\beta_{i})^{2}},\quad \Gamma_{3}(M_{i})=2\frac{\alpha_{i}\beta_{i}^{2} }{(1-\beta_{i})^{3}}, \quad i=1,2.
\end{align*}
Besides, it is well known that $ \Gamma_{1}(W)=m_{1}, \Gamma_{2}(W)=\left(m_{2}-m_{1}^{2}-m_{1}\right), \Gamma_{3}(W)=m_{3}-3m_{1}m_{2}+2m_{1}^{3}-3m_{2}+3m_{1}^{2}+2m_{1},$ where $m_{j}$ be the $j$-th origin moment of $W$.

The basic principle is to guarantee the first three moments of both  $W$ and  $M_{i}$ match each other, namely, that is
\begin{align}\label{1.10}
   \Gamma_{1}(M_{i})=\Gamma_{1}(W),\quad\Gamma_{2}(M_{i})=\Gamma_{2}(W),\quad\Gamma_{3}(M_{i})=\Gamma_{3}(W),\quad i=1,2.
\end{align}
In particular,  the parameters $\{n, p, \lambda\}$ in $M_1$ and  $\{r, \bar{p}, \lambda\}$ in $M_2$ are delicately chosen to  satisfy the requirments
\begin{align}\label{111}
     n&=\lfloor\frac{{4\Gamma_{2}(W)}^{3}}{{\Gamma_{3}(W)}^{2}}\rfloor, \quad p=-\frac{\Gamma_{3}(W)}{2\Gamma_{2}(W)},\quad\lambda=\mathbb{E}W-np,\quad \delta=\frac{{4\Gamma_{2}(W)}^{3}}{{\Gamma_{3}(W)}^{2}}-n;\\\label{111'}
     r&=\frac{4\Gamma_{2}(W)^{3}}{\Gamma_{3}(W)^{2}},  \quad\quad p = \frac{2\Gamma_{2}(W)}{2\Gamma_{2}(W) +\Gamma_{3}(W)}, \quad\quad \lambda =\mu - \frac{rq}{p}.
\end{align}
\section{Proofs of Main Results}\label{S3}

\begin{proof}[Proof of Theorem \ref{main}]
Denote $\bar{\beta_{1}}=1-\beta_{1}$ and $\bar{\beta_{2}}=1-\beta_{2}$. Let us begin with the proof of (i). From Lemma \ref{r}, we have
$$\mathcal{A}_{M_{2}} g(k)=(a_{2}+\beta_{2} k) g(k+1)-k g(k) - \lambda \beta_{2} \Delta g(k+1),$$
where $a_{2}=\alpha_{2}+\lambda \bar{\beta_{2}}.$ Taking expectation with respect to $W$ yields
\begin{align*}
\mathbf{E}\mathcal{A}_{M_{2}} g(W)&= \mathbf{E}(a_{2}+\beta_{2} W)g(W+1)-\mathbf{E}W g(W) - \lambda \beta_{2} \mathbf{E}\Delta g(W+1) \\
     =&\bar{\beta_{2}}\left[\frac{a}{\bar{\beta_{2}}} \mathbf{E}[g(W+1)]-\mathbf{E}[W g(W)]\right]+\beta_{2} \mathbf{E}[W \Delta g(W)] \notag\\
&- \lambda \beta_{2} \mathbf{E}\Delta g(W+1).\notag
\end{align*}
Since $a_{2}/\bar{\beta_{2}} = \mathbf{E}W=\sum\nolimits_{i \in J}  \mathbf{E} X_{i}$, then
\begin{align*}
   \mathbf{E}\left[\mathcal{A}_{M_{2}} g(W)\right]&=\bar{\beta_{2}}\Big[\sum\nolimits_{i \in J}  \mathbf{E} X_{i} \mathbf{E}[g(W+1)]-\sum\nolimits_{i \in J}  \mathbf{E}\left[X_{i} g(W)\right]\Big]\\
  &+ \beta_{2}\mathbf{E}[W \Delta g(W)] - \lambda \beta_{2} \mathbf{E}\Delta g(W+1).\notag
\end{align*}
Set
$$W_{i} =W-X_{A_{i}}^{*},\quad W_{i}^{*} =W-X_{B_{i}}^{*},\quad W_{i}^{**} =W-X_{C_{i}}^{*}.$$
Using the independence of  $X_{i}$ and $W_{i}$, we obtain
\begin{align}\label{132}
  \quad\mathbf{E}\left[\mathcal{A}_{M_{2}} g(W)\right]&=\bar{\beta_{2}} \sum\nolimits_{i \in J} \mathbf{E} X_{i} \mathbf{E}\left[\sum\nolimits_{j=1}^{X_{A_{i}}^{*}} \Delta g\left(W_{i}+j\right)\right]- \lambda \beta_{2} \mathbf{E}\Delta g(W+1)\notag \\
  &-\bar{\beta_{2}} \sum\nolimits_{i \in J}  \mathbf{E}\left[X_i \sum\nolimits_{j=1}^{X_{A_{i}}^{*}-1} \Delta g\left(W_{i}+j\right)\right]+\beta_{2} \sum\nolimits_{i \in J}  \mathbf{E}\left[X_{i} \Delta g(W)\right].
\end{align}
 A crucial observation is
\begin{align}\label{0}
& \bar{\beta_{2}}\left\{\sum\nolimits_{i \in J}  \mathbf{E} X_{i}  \mathbf{E}\left[\sum\nolimits_{j=1}^{X_{A_{i}}^{*}} 1\right]-\sum\nolimits_{i \in J}  \mathbf{E}\left[X_{i} \sum\nolimits_{j=1}^{X_{A_{i}}^{*}-1} 1\right]\right\}+\beta_{2}\sum\nolimits_{i \in J}  \mathbf{E} X_{i}\\
=&\sum\nolimits_{i \in J}  \mathbf{E} X_{i}-\bar{\beta_{2}}\sum\nolimits_{i \in J,j \in A_{i}} \left(\mathbf{E}X_{i} X_{j}-\mathbf{E} X_{i} \mathbf{E} X_{j}\right) = \mathbf{E}W-\bar{\beta_{2}}\operatorname{Var}W.\notag
\end{align}
Note that for $M_{2}$,
$\mathbf{E}W =\alpha_{2}/\bar{\beta_{2}}+\lambda, \operatorname{Var}W = \alpha_{2}/\bar{\beta_{2}}^{2}+\lambda.$ So (\ref{0}) is equal to say
\begin{align}\label{1}
     \lambda \beta_{2} = \bar{\beta_{2}}\Big\{\sum\nolimits_{i \in J}  \mathbf{E} X_{i} \mathbf{E}\left[X_{A_{i}}^{*}\right]-\sum\nolimits_{i \in J}  \mathbf{E}\left[X_{i}\left(X_{A_{i}}^{*}-1\right)\right]\Big\}+\beta_{2} \sum\nolimits_{i \in J}  \mathbf{E} X_{i}.
\end{align}
Substituting (\ref{1}) into (\ref{132}),  noting the independence of $X_{A_{i}}^{*}$ and $W_{i}^{*}$ and  the following elementary second-order  difference identity,
  \begin{align}\label{cha}
    \Delta g\left(W_{i}+m\right)-\Delta g\left(W_{i}^{*}+1\right) = \sum\nolimits_{\ell=1}^{X_{B_{i}\backslash A_{i}}^{*} +m-1} \Delta^{2} g\left(W_{i}^{*}+\ell\right),\quad m\in \mathbf{Z},
  \end{align} we obtain
  \begin{align}
  \mathbf{E}\left[\mathcal{A}_{M_{2}} g(W)\right] &= -\sum\nolimits_{i \in J} \Big\{\bar{\beta_{2}} \left[\mathbf{E}X_{i} \mathbf{E}X_{A_{i}}^{*}-\mathbf{E}X _ { i }( X_{A_{i}}^{*}-1)\right]+ \beta_{2}\mathbf{E}X_{i}\Big\}
  \notag\\
&\times\mathbf{E}\left[\sum\nolimits_{\ell=1}^{X_{B_{i}}^{*}} \Delta^{2} g\left(W_{i}^{*}+\ell\right)\right]+ \beta_{2}\sum\nolimits_{i \in J}  \mathbf{E}\left[X_{i} \sum\nolimits_{\ell=1}^{X_{B_{i}}^{*}-1} \Delta^{2} g\left(W_{i}^{*}+\ell\right)\right] \notag\\
&+\bar{\beta_{2}} \sum\nolimits_{i \in J}  \mathbf{E} X_{i} \mathbf{E}\left[\sum\nolimits_{j=1}^{X_{A_{i}}^{*}} \sum\nolimits_{\ell=1}^{X_{B_{i}\backslash A_{i}}^{*} +j-1} \Delta^{2} g\left(W_{i}^{*}+\ell\right)\right] \notag\\
  &-\bar{\beta_{2}} \sum\nolimits_{i \in J}  \mathbf{E}\left[X_{i} \sum\nolimits_{j=1}^{X_{A_{i}}^{*}-1} \sum\nolimits_{\ell=1}^{X_{B_{i}\backslash A_{i}}^{*} +j-1} \Delta^{2} g\left(W_{i}^{*}+\ell\right)\right]:=\Upsilon_{2}. \notag
\end{align}
\normalsize
A {\it nice} coincidence is that $\Upsilon_{2}$ above with  $\Delta^{2} g\left(W_{i}^{*}+\ell \right)$ replaced by 1 vanishes. For clarity, we state it as a lemma.

\begin{lemma} \label{coinlm}Let
\begin{align*}
\Upsilon_{20}:=&\sum\nolimits_{i \in J}\Big\{\bar{\beta_{2}}\mathbf{E} X_{i} \mathbf{E}(X_{A_{i}}^{*}(X_{B_{i}\backslash A_{i}}^{*} +j-1))
  -   \mathbf{E} [X_{i} (X_{A_{i}}^{*}-1)(X_{B_{i}\backslash A_{i}}^{*} +j-1) ]+\notag\\
  &\beta_{2}  \mathbf{E} X_{i}  (X_{B_{i}}^{*}-1)  -\Big\{\bar{\beta_{2}}  \left[\mathbf{E}X_{i} \mathbf{E}X_{A_{i}}^{*}-\mathbf{E}X _ { i }( X_{A_{i}}^{*}-1)\right]+ \beta_{2}\mathbf{E}X_{i}\Big\} \mathbf{E} X_{B_{i}}^{*}\Big\},
\end{align*}
then $\Upsilon_{20}=0$
\end{lemma}
\begin{proof}

 it is easy to see
 \begin{align}  \label{ca}
 \Upsilon_{20} = \sum\nolimits_{i \in J}&\left\{[\mathbf{E}X_{i}\mathbf{E}X_{A_{i}}^{*}X_{B_{i}}^{*} - \mathbf{E}X_{i}X_{A_{i}}^{*}X_{B_{i}}^{*} - \mathbf{E}X_{i}\mathbf{E}X_{A_{i}}^{*}\mathbf{E}X_{B_{i}}^{*} + \mathbf{E}X_{i}X_{A_{i}}^{*}\mathbf{E}X_{B_{i}}^{*}\right.\notag\\
  &  \left.+\frac{1}{2}\left(\mathbf{E}X_{i}(X_{A_{i}}^{*})^{2}-\mathbf{E}X_{i}(\mathbf{E}(X_{A_{i}}^{*})^{2}+  \mathbf{E}X_{i}X_{A_{i}}^{*} - \mathbf{E}X_{i}\mathbf{E}X_{A_{i}}^{*}\right)]\bar{\beta_{2}}\right.
  \notag\\
  &  \left.+\mathbf{E}X_{i}X_{B_{i}}^{*}-\mathbf{E}X_{i}\mathbf{E}X_{B_{i}}^{*}-\mathbf{E}X_{i} \right\}.
\end{align}
\normalsize
By local independence, we have the following two equality:
\begin{align}\label{zx}
   \sum\nolimits_{i \in J}(\mathbf{E}X_{i}X_{B_{i}}^{*}-\mathbf{E}X_{i}\mathbf{E}X_{B_{i}}^{*})=\sum\nolimits_{i \in J}(\mathbf{E}X_{i}X_{A_{i}}^{*}-\mathbf{E}X_{i}\mathbf{E}X_{A_{i}}^{*}) = m_{2}-m_{1}^{2};
\end{align}
\begin{align}\label{xz}
&\mathbf{E}X_{i}\mathbf{E}X_{A_{i}}^{*}W - \mathbf{E}X_{i}X_{A_{i}}^{*}W - \mathbf{E}X_{i}\mathbf{E}X_{A_{i}}^{*}\mathbf{E}W + \mathbf{E}X_{i}X_{A_{i}}^{*}\mathbf{E}W\\
&= \mathbf{E}X_{i}\mathbf{E}X_{A_{i}}^{*}X_{B_{i}}^{*} - \mathbf{E}X_{i}X_{A_{i}}^{*}X_{B_{i}}^{*} - \mathbf{E}X_{i}\mathbf{E}X_{A_{i}}^{*}\mathbf{E}X_{B_{i}}^{*} + \mathbf{E}X_{i}X_{A_{i}}^{*}\mathbf{E}X_{B_{i}}^{*}.\notag
\end{align}
Applying (\ref{zx}) and (\ref{xz}), (\ref{ca}) becomes
\begin{align}\label{12345}
&\frac{\bar{\beta_{2}}}{2}\sum\nolimits_{i \in J}[\mathbf{E}X_{i}(X_{A_{i}}^{*})^{2}-\mathbf{E}X_{i}\mathbf{E}(X_{A_{i}}^{*})^{2}+2(\mathbf{E}X_{i}\mathbf{E}X_{A_{i}}^{*}W - \mathbf{E}X_{i}X_{A_{i}}^{*}W)]\\
 & +(\frac{\bar{\beta_{2}}}{2}+1)(m_{2}-m_{1}^{2})-m_{1} + \bar{\beta_{2}}(m_{2}-m_{1}^{2})m_{1}.\notag
\end{align}
In addition, noting $W=X_{A_{i}}^{*}+X_{A_{i}^{c}}^{*}$, it follows
\begin{align}\label{10}
&\mathbf{E}X_{i}(X_{A_{i}}^{*})^{2}-\mathbf{E}X_{i}\mathbf{E}(X_{A_{i}}^{*})^{2}+2(\mathbf{E}X_{i}\mathbf{E}X_{A_{i}}^{*}W - \mathbf{E}X_{i}X_{A_{i}}^{*}W) \\
   = &-\mathbf{E}X_{i}(X_{A_{i}}^{*})^{2}+\mathbf{E}X_{i}\mathbf{E}(X_{A_{i}}^{*})^{2}- 2\mathbf{E}X_{i}X_{A_{i}}^{*}X_{A_{i}^{c}}^{*}+2\mathbf{E}X_{i}\mathbf{E}X_{A_{i}}^{*}X_{A_{i}^{c}}^{*}\notag\\
   =  &-\mathbf{E}X_{i}W^{2}+\mathbf{E}X_{i}\mathbf{E}W^{2},\notag
\end{align}
where we used the independence between $X_{i}$ and $X_{A_{i}^{c}}^{*}$ in the last equation.

Combining equations (\ref{12345}) and (\ref{10}), we end up with that (\ref{ca}) becomes
\begin{align}\label{3160}
(\bar{\beta_{2}}2/+1)(m_{2}-m_{1}^{2})-m_{1} + \bar{\beta_{2}}(m_{2}-m_{1}^{2})m_{1} - \bar{\beta_{2}}(m_{3}-m_{1}m_{2})/2.
\end{align}
Solving (\ref{1.10}) yields
\begin{align}\label{3.4}
\beta_{2}=\frac{m_{3}-3m_{1}m_{2}-3m_{2}+2m_{1}^{3}+3m_{1}^{2}+2m_{1}}{m_{3}-3m_{1}m_{2}-m_{2}+2m_{1}^{3}+m_{1}^{2}}.
\end{align}
Substituting (\ref{3.4}) into (\ref{3160}), we conclude  that  $\Upsilon_{20}=0$, as desired.
\end{proof}

Proceed with the proof of Theorem \ref{main}. Applying Lemma \ref{coinlm}, noting the independence of $X_{B_{i}}^{*}$ and $W_{i}^{**}$ and the following third-order  difference identity
  $$\Delta^{2} g\left(W_{i}^{*}+m\right)-\Delta^{2} g\left(W_{i}^{**}+1\right) = \sum\nolimits_{k=1}^{X_{C_{i}\backslash B_{i}}^{*} +m-1} \Delta^{3} g\left(W_{i}^{*}+k\right),\quad m\in \mathbf{Z},$$ some simple algebra lead to
  \begin{align*}
   \left|\mathbf{E}\mathcal{A}_{M_{2}} g(W)\right| \leq& \sum\nolimits_{i \in J}\left[\sup_{k\in \mathbf{Z}^{+}}\operatorname{ess}\mathbf{E}[\Delta^{3}g(W+k)\mid X_{C_{i}}]\right]\\
   &\times \sum\nolimits_{(\mathbf{E})}[\bar{\beta_{j}}\mathbf{E}X_{i}(\mathbf{E})X_{A_{i}}^{*}(\mathbf{E})X_{B_{i}}^{*}(\mathbf{E})X_{C_{i}}^{*}
    +\beta_{j}\mathbf{E}X_{i}(\mathbf{E})X_{B_{i}}^{*}(\mathbf{E})X_{C_{i}}^{*})].
\end{align*}

It is easy to verify that $\mathbf{E}[\Delta^{3}g(W+k)\mid X_{C_{i}}] \leq \|\Delta g\| D(W\mid X_{C_{i}}),$
where  $D(W\mid X_{C_{i}})$ was defined in (\ref{10013}).
Using the definition of    $\Xi_{i,2}$ and some apparent comparsion, the last inequation becomes $\left|\mathbf{E}\mathcal{A}_{M_{2}} g(W)\right| \leq \|\Delta g\|\sum\nolimits_{i \in J}\Xi_{i,2},$ which,  together with  (\ref{888}) implies  (\ref{i}) holds.

Turn to the proof of (ii). Following the proof of (i),
 the R.H.S of (\ref{0}) is equal to
$\mathbf{E}W - \bar{\beta_{1}}\operatorname{Var}W=\lambda\beta_{1}+ \delta \beta_{1}^{2}/\bar{\beta_{1}}.$ Thus we immediately have
\begin{align}\label{wxl}
    \mathbf{E}\mathcal{A}_{M_{1}} g(W) \leq \Upsilon_{1} +\frac{ \delta \beta_{1}^{2}}{ \bar{\beta_{1}} }  \mathbf{E}\Delta g(W+1),
\end{align}
where $\Upsilon_{1}$ is obtained by replacing $\beta_{2}$ with $\beta_{1}$ from $\Upsilon_{2}.
$ Noting (\ref{1.10}), an analog of  the proof of (\ref{i})  yields
\begin{align}\label{322}
    \left|\mathbf{E}\mathcal{A}_{M_{1}} g(W)\right| \leq \|\Delta g\|\left(\sum\nolimits_{i\in J}\Xi_{i,1}  + \frac{\delta p^{2}}q\right) .
\end{align}
Finally, using (\ref{999}) and (\ref{322}), the proof is completed.
\end{proof}
\section{Applications}\label{S4}
\subsection{$(k_{1}, k_{2})$-runs}
Fix $k_{1}, k_{2} \geq 1$, let $m=k_{1}+k_{2}-1$ and $J =\{1,2,\cdots,Nm\}$. Suppose that $\xi_{1}, \xi_{2}, \cdots ., \xi_{Nm}$ are a sequence of independent Bernoulli random variables with $\mathbf{P}(\xi_{j}=1)=p$ for $j\in J$.
 We say that a $(k_{1}, k_{2})$-event  if there occurs $k_{1}$ consecutive $0'$s followed by $k_{2}$ consecutive $1'$s. To avoid edge effects, we identify $i + Nmj$ as $i$    for $ i \in J, j\in \mathbf{Z}$\;. Define
$$X_{j}=\left(1-\xi_{j}\right) \cdots\left(1-\xi_{j+k_{1}-1}\right) \xi_{j+k_{1}}\cdots \xi_{j+k_{1}+k_{2}-1}.$$
It is easy to see that the  local dependence structure is satisfied with
\begin{align*}
A_{i} &=\{j\in J:|j-i| \leq m\};  B_{i} =\{j:|j-i| \leq 2m\}; C_{i} =\{j:|j-i| \leq 3m\}.
\end{align*}
Denote by $W_{N}=\sum\nolimits_{j=1}^{Nm}X_{j}$ the number of occurrences of $(k_{1},k_{2})$-events in $Nm$ trials, which is often called a modified binomial distribution. For ease of notation, we henceforth suppress the dependence of quantities on $N$ when it is clear from the context. \\
Let $b=\mathbf{E}X_{j}=(1-p)^{k_{1}}p^{k_{2}}.$ It follows from some simple but tedious calculations,
\begin{align*}
  \Gamma_{1}(W)=Nmb,\quad\Gamma_{2}(W)=-Nm(2m+1)b^{2},\quad
  \Gamma_{3}(W)=Nm[9m^{2}+9m+2]b^{3}.
\end{align*}

Since  $\Gamma_{2}(W)$  is negative, we use $M_1=B(n,p)*\mathcal{P}(\lambda)$ to approximate $W$ by matching the first three moments.
Following the argument in Section \ref{S3}, we get
\begin{align}\label{nmd}
 n=\lfloor \frac{Nm(2m+1)^{3}}{[(2m+1)^{2}+\frac{m(m+1)}{2}]^{2}}\rfloor,\; p=\left[2m+1+\frac{m(m+1)}{2(2m+1)}\right]b, \;\lambda = Nmb - np.
\end{align}

\begin{theorem}\label{T42}
Let $m\geq 1$ and $\left\{p_{i},i\in J\right\}$ are identical to $p$, assume
\begin{align}\label{726}
b:=\left(1-p\right)^{k_{1}}p^{k_{2}} < \frac{2(2m+1)}{2(2m+1)^{2}+3m(m+1)}:=c_{m}.
\end{align}
Then we have
\begin{align}\label{410}
    &d_{T V}(W, M_1) \leq O(N^{-1});
\end{align}
and
\begin{align}
\label{014}
     \sup_{k\in \mathbf{Z}}& \Big|\mathbf{P}(W\leq k)-\mathbf{\Phi}_{3}(\frac{k+1/2-\mu}{\sigma})\Big|\leq O(N^{-1}).
\end{align}
\end{theorem}
    We claim that (\ref{410}) is a significant improvement over the bound given in the literature. In fact, Theorem 5.2 of Barbour \cite{barbour1999poisson}  and Example 2.1 of Kumar [18] only attain $O(N^{-1/2})$, where they used the two-parameter compound Poisson distribution and pseudo-binomial distribution to approximate $W$ under $m=1$, respectively; and Theorem 3.1 of \cite{upadhye2018pseudo} by using Poisson approximation, which is of order $O(N^{-1/2})$ under $m\geq 2$.

Define
$T_{i} =\sum\nolimits_{j=(i-1) m+1}^{i m} X_{j}$ for $i=1,2, \cdots, N.$ Let $J^{\prime}=\{1,2, \cdots, N\}$. Hence $T_{1}, T_{2}, \cdots, T_{N}$ are 1-dependent random variables, which means the local  dependence structure is satisfied with
$$A_{i}^{\prime}=\{j:|j-i| \leq 1\} \cap J^{\prime},\quad B_{i}^{\prime}=\{j:|j-i| \leq 2\} \cap J^{\prime},\quad C_{i}^{\prime}=\{j:|j-i| \leq 3\} \cap J^{\prime}.$$
Define $Z := \left\{T_{2 k}, 1 \leq k \leq\lfloor N / 2\rfloor\right\}, N(i,Z)= \left|T_{J^{\prime}\backslash C_{i}^{\prime}}\backslash Z \right|$ where  $\left| A\right|$ stands for the cardinality of $A$ and the specific value of $N(i,Z)$ is $O(N)$. Set $\mathcal{F}(i,Z) = \{1,\cdots,N(i,Z)\}$. Note $T_{J^{\prime}\backslash C_{i}^{\prime}}\backslash Z$ can be written as $\{T_{l_{j}}^{Z},j\in \mathcal{F}(i,Z),l_{1} < l_{2} <\cdots<l_{N(i,Z)}\}$ where $T_{l_{j}}^{Z}$ satisfies $\mathcal{L}(T_{l_{j}}^{Z})  = \mathcal{L}(T_{l_{j}}|Z,X_{C_{i}})$.  Hence these $(T_{l_{j}}^{Z}, j\in \mathcal{F}(i,Z))$ are independent of each other. Fixing $Z=z\in \{0,1\}^{\lfloor N / 2\rfloor}$,  $\mathcal{L}\left(W_{i}^{**} \mid Z=z, X_{C_{i}}\right)$ can be represented as the sum of independent random variables  $\{T_{l_{j}}^{z}, j\in \mathcal{F}(i,z)\}$.  Thus it follows
 $$\big((W_{i})^{**} \mid Z, X_{C_{i}}\big) \overset{d}{=} \Big(\sum\nolimits_{j\in J} T_{l_{j}} \mid Z, X_{C_{i}}\Big)\overset{d}{=} \Big(\sum\nolimits_{j\in J}  T_{l_{j}}^{Z}\Big).$$
 Noting that $D\left(W_{i}^{**}\mid T_{C_{i}^{\prime}}\right) =D(W \mid X_{C_{i}})$, it suffices to control the $D\left(W_{i}^{**}\mid T_{C_{i}^{\prime}}\right)$.
We define
$$\mathcal{D}(X)=\left\|\mathcal{L}(X) *\left(I_{1}-I_{0}\right)\right\|_{\mathrm{TV}}=\sum\nolimits_{k=0}^{\infty}\left|P(X=k+1)-P(X=k)\right|,$$
$$V_{Z}=\sum\nolimits_{j \in  \mathcal{F}(i,Z)}  \left[1 / 2\wedge\big(1-\mathcal{D}(T_{l_{j}}^{Z})/2\big)\right],\quad v_{Z}^{*}=\max_{j \in  \mathcal{F}(i,Z)}\Big\{1 / 2\wedge\big(1-\mathcal{D}(T_{l_{j}}^{Z})/2\big)\Big\}.$$

According to (5.12) of \cite{rollin2008symmetric} and (4.9) of \cite{barbour2002total}, we have
 \begin{align}\label{Fii}
   D\left(W_{i}^{**}\mid T_{C_{i}^{\prime}}\right) &\leq \mathbf{E}\{\mathbf{E}[ D\left(W_{i}^{**}\right)|Z, T_{C_{i}^{\prime}}]\}  \leq 4\mathbf{E} \Big\{1 \wedge \frac{2}{\left(V_{Z}-4  v_{Z}^{*}\right)_{+}}\Big\}=O(N^{-1}).
 \end{align}
  \begin{proof}[Proof of  Theorem \ref{T42}]

  As for (\ref{410}), in this case we get
  $$\lim_{N \rightarrow \infty} \theta_{1} = \frac{m(m+1)b}{2(2m+1)-\left[2(2m+1)^{2}+m(m+1)\right]b}$$
from (\ref{theta}) and (\ref{nmd}). We claim $\lim_{N \rightarrow \infty}\theta_{1} < 1/2$ under the assumption (\ref{726}). In fact,
$$\lim_{N \rightarrow \infty} \theta_{1} = \frac{m(m+1)b}{2(2m+1)-\left[2(2m+1)^{2}+m(m+1)\right]b}.$$
Notice that $f(x) = m(m+1)x\left\{2(2m+1)-[2(2m+1)^{2}+m(m+1)]x\right\}^{-1} $ is monotonic increasing and $f(c_{m})=1/2$. So, $\lim_{N \rightarrow \infty} \theta_{1} = f(b) < f(c_{m}) =1/2.$

Observably,  $\delta\in [0,1)$  and (\ref{Fii}) implies $\sum_{i\in J} \Xi_{i,1}\leq O(N)$. Then substituting  (\ref{Fii})  into Theorem \ref{main} directly yields (\ref{410}).

The proof of \eqref{014} is an immediate result of \eqref{410} and Theorem \ref{Main2}.
\end{proof}
\subsection{$k$-runs}\label{kkuns}
In this subsection, we turn to another special case with $k_{1}=0$ and $k_{2}=k>0$. It is easy to see that the  local dependence structure is satisfied with
$$A_{i} =\{j:|j-i| \leq k\} \cap J,\quad B_{i} =\{j:|j-i| \leq 2k\} \cap J,\quad
C_{i} =\{j:|j-i| \leq 3k\} \cap J.$$

Denote  $W_{N,k}= \sum\nolimits_{j=1}^{N}X_{j}$, which is often termed as $k$-runs.  As before, we suppress the dependence of quantities on $N$ and $k$ in the sequel.  We focus on the case that $p$ and $k$ depend on $N$ with $k \log p \rightarrow -\infty$. The expectation of $W$ is $Np^{k}$, so is $\Gamma_{1}(W)$. We omit the cumbersome calculations and give  asymptotic expressions of the second and third factorial cumulants:
\begin{align}\label{33}
   \Gamma_{2}(W)= \Big[\frac{2p}{1-p}+O(p^k)\Big] Np^{k},\quad \Gamma_{3} (W)=\left[\frac{6p^{2}}{(1-p)^{2}}+O(p^{k})\right]Np^{k}.
\end{align}

From (\ref{111'}) and (\ref{33}), it is easy to obtain
\begin{align}\label{1311}
    r=O(Np^{k}), \quad\quad\quad\quad \bar{p} =O(1), \quad\quad\quad\quad
\lambda=O(Np^{k})
\end{align}

Moreover, by (\ref{33}), and the definition (\ref{theta}) of $\theta_{2}$, $\lim_{N \rightarrow \infty} 2(1-p)\theta_{2}/p =1.$ Hence we assume further $p<1/2$ to guarantee that $\theta_{2}<1/2$. Our result reads.
\begin{theorem}\label{T5}
 Assume $p < 1/2$, $N>10k$ and $k \log p \rightarrow -\infty$ as $N \rightarrow \infty$. Then
\begin{align}\label{zongjie}
    d_{\mathrm{TV}}\left(W, M_2\right) \leq  \left(9 \wedge \frac{95.22(1-2p)}{(N-10 k+8) p^{k}(1-p)^{2}}\right)(2k-1)(4k-3)(6k-5) p^{3}.
\end{align}
\normalsize

\end{theorem}
Assume $p$ and $k$ are numeric constants (do not depend on $N$), the gap  between $W$ and $M_2$ using Theorem \ref{T5} is actually $O(N^{-1})$, which is an improvement of Corollary 1.1 of \cite{wang2008negative} that attains  $O(N^{-1/2})$. Moreover, in the case that $p$ and $k$ depend on $N$, Corollary 1.1 of  Wang and Xia \cite{wang2008negative} showed
\begin{align}\label{zj}
  d_{\mathrm{TV}}(W, NB(r^{\prime},p^{\prime})) \leq 4.5(4 k-3)(2 k-1) p^{2}\Big(2 \wedge \frac{4.6}{\sqrt{(N-4 k+2) p^{k}(1-p)^{3}}}\Big),
\end{align}
where $r^{\prime} = \Gamma_{1}(W)^{2}/\Gamma_{2}(W)$, $p^{\prime} = \Gamma_{1}(W)/[\Gamma_{1}(W)+\Gamma_{2}(W)]$.  The error upper bound (\ref{zongjie}) is in general better than (\ref{zj}).

Before proving Theorem \ref{T5}, we introduce some additional notation and give a lemma. Let
$N_{1} = \lfloor (N-k)/2\rfloor, \;W_{1}=\sum\nolimits_{j=1}^{N_{1}}X_{j},\; W_{2}=\sum\nolimits_{j=N_{1}+k}^{N}X_{j}.$ Denote $C_{1,i}=C_{i}\cap \{1,\cdots ,N_{1}\}, \;C_{2,i}=C_{i}\cap \{N_{1}+k,\cdots, N\}.$ Without loss of generality, we assume  $\{N_{1}+1, \cdots, N_{1}+k-1\} \in C_{i}$ and $|C_{1,i}|=2.5k-3$, $|C_{2,i}|=2.5k-2$.
\begin{lemma}\label{llee}
    For $N > 10k$, we have
    \begin{align}\label{a1}
        \mathcal{D}(W_{j}|X_{C_{j,i}}) \leq 1 \wedge \frac{2.3}{\sqrt{\left(0.5N-5k+4\right)\left(1-p\right)^{3} p^{k}}},\quad j=1,2.
    \end{align}
\end{lemma}
\begin{proof}
        We start with the proof of (\ref{a1}) with $j=1$. Note that $X_{C_{1,i}}$ contains $(2.5k-3)$ consecutive elements from $Y_{|N_{1}|-2.5k+2}$ to $Y_{|N_{1}|}$ which determined by $\{\xi_{|N_{1}|-2.5k+2}$, $\xi_{|N_{1}|-2.5k+3}, \cdots, \xi_{|N_{1}|+k-1}\}.$

         Denote $G=\{1,2,\cdots,N_{1}\}$, $H = \{j:|N_{1}|-2.5k+2 \leq j \leq |N_{1}|+k-1\} \cap J$. Define $\gamma_{l}:=\mathbf{P}(\xi_{l}=1\mid X_{C_{1,i}}), l\in J$. It is easy to find for $l \notin H$, $\gamma_{l}=p$.  Using Lemma 2.1 of \cite{wang2008negative}, we obtain
          \begin{align*}
        \mathcal{D}(W_{1}|X_{C_{1,i}}) \leq 1 \wedge \frac{2.3}{\sqrt{\left(0.5N-5k+4\right)\left(1-p\right)^{3} p^{k}}}
    \end{align*}
           By the same token, we complete the proof of (\ref{a1}) with $j=2$.
\end{proof}
\begin{proof}[Proof of Theorem \ref{T5}]
We start with the calculation of $D(W|X_{C_{i}})$.
Notice that $D(W|X_{C_{i}})=D(W_{1}+W_{2}|X_{C_{i}})$ and
$\mathcal{L}(W_{1}+W_{2}|X_{C_{i}})=\mathcal{L}(W_{1}|X_{C_{1,i}})*\mathcal{L}(W_{2}|X_{C_{2,i}}).$
 So by (1.2)  of \cite{rollin2008symmetric} and Lemma \ref{llee} we obtain
\begin{align*}
D(W|X_{C_{i}}) &\leq \mathcal{D}(W_{1}|X_{C_{1,i}}) \mathcal{D}(W_{2}|X_{C_{2,i}})\leq 1 \wedge \frac{10.58}{(N-10k+8) p^{k}(1-p)^{3}}.\notag
\end{align*}
\normalsize

Next, using (\ref{theta}), (\ref{1311}), the definition of $\Xi_{i,,2}^{\prime}$ and observing $\bar{q}\leq 3p\bar{p}/2$ when $p <1/2$ , we get
$$\sum\nolimits_{i=1}^{n}\Theta_{2} \Xi_{i,2}^{\prime} \leq  9(2k-1)(4k-3)(6k-5)\frac{(1-p)p^3}{1-2p}.$$
Finally, we use (i) of Theorem \ref{main} to conclude the proof.
\end{proof}


\bibliographystyle{siamplain}
\bibliography{references}
\end{document}